\documentclass[11pt]{amsart}
\usepackage{amsmath}
\usepackage{amssymb}
\usepackage{amsthm}
\usepackage{latexsym,tikz}
\usepackage{hyperref}
\usepackage{enumerate}
\usepackage{xcolor}
\usepackage{calrsfs}
\DeclareMathAlphabet{\pazocal}{OMS}{zplm}{m}{n}

\newcommand{\conv}{{conv}}

\newcommand{\R}{\mathbb{R}}
\newcommand{\bR}{\mathbb{R}}
\newcommand{\C}{\mathbb{C}}
\newcommand{\cnt}{\mathcal{C}}

\newcommand{\D}{\mathbb{D}}
\newcommand{\bH}{\mathbb{H}}
\newcommand{\Hq}{\mathbb{H}}
\newcommand{\bP}{\mathbb{P}}
\newcommand{\bS}{\mathbb{S}}
\newcommand{\M}{\mathcal{M}}

\newcommand{\vc}[1]{\boldsymbol{#1}}
\newcommand{\interior}[1]{{\kern0pt#1}^{\mathrm{o}}}

 \newtheorem{thm}{Theorem}[section]
 \newtheorem{cor}[thm]{Corollary}
 
 \newtheorem{lemma}[thm]{Lemma}

 \newtheorem{prop}[thm]{Proposition}
 \theoremstyle{definition}
 \newtheorem{defn}[thm]{Definition}
 \theoremstyle{remark}

 \newtheorem*{ex}{Example}

 \newcommand*{\myproofname}{\S\nopunct}

 \numberwithin{equation}{section}

\begin{document}
\setlength{\abovedisplayskip}{3pt}
\setlength{\belowdisplayskip}{3pt}

\thanks{The second author was partially supported by FCT
 through project UID/MAT/04459/2013 and the third author was partially supported by FCT through CMA-UBI, project PEst-OE/MAT/UI0212/2013.}

\title[The star-center of the quaternionic numerical range]{The star-center of the quaternionic numerical range}

\author[L. Carvalho]{Lu\'{\i}s Carvalho}
\address{Lu\'{\i}s Carvalho, ISCTE - Lisbon University Institute\\    Av. das For\c{c}as Armadas\\     1649-026, Lisbon\\   Portugal}
\email{luis.carvalho@iscte-iul.pt}
\author[Cristina Diogo]{Cristina Diogo}
\address{Cristina Diogo, ISCTE - Lisbon University Institute\\    Av. das For\c{c}as Armadas\\     1649-026, Lisbon\\   Portugal\\ and \\ Center for Mathematical Analysis, Geometry,
and Dynamical Systems\\ Mathematics Department,\\
Instituto Superior T\'ecnico, Universidade de Lisboa\\  Av. Rovisco Pais, 1049-001 Lisboa,  Portugal
}
\email{cristina.diogo@iscte-iul.pt}
\author[S. Mendes]{S\'{e}rgio Mendes}
\address{S\'{e}rgio Mendes, ISCTE - Lisbon University Institute\\    Av. das For\c{c}as Armadas\\     1649-026, Lisbon\\   Portugal\\ and Centro de Matem\'{a}tica e Aplica\c{c}\~{o}es \\ Universidade da Beira Interior \\ Rua Marqu\^{e}s d'\'{A}vila e Bolama \\ 6201-001, Covilh\~{a}}
\email{sergio.mendes@iscte-iul.pt}

\date{\today}

\maketitle

\begin{abstract}
In this paper we prove that the quaternionic numerical range is always star-shaped and its star-center is given by the equivalence classes of the star-center of the bild. We determine the star-center of the bild, and consequently of the numerical range, by showing that the geometrical shape of the upper part of the center is defined by two lines, tangents to the lower bild.
\end{abstract}


\maketitle
\section{Introduction}

Let $\Hq$ denote the skew-field of Hamilton quaternions. Let $A$ be a $n\times n$ matrix with quaternionic entries. It is well known that the numerical range $W_{\Hq}(A)=W(A)$ is a connected but not necessarily convex subset of the quaternions. The group of unitary quaternions $\bS_{\Hq}$ acts on $\Hq$ by automorphisms. Since every class $[q]$, $q\in \Hq$, has a representative in $\C^+$ and each class of $q\in W(A)$ is contained in $W(A)$, it became clear from the early studies of the quaternionic numerical range that it is enough to study the bild of $A$, $B(A)=W(A)\cap \C$ or the upper-bild $B^+(A)=W(A)\cap \C^+$.
 The latter has the advantage of being always convex whereas $B(A)$ is convex if, and only if, $W(A)$ is convex, see \cite[page 53]{Zh} and theorem \ref{theo convex W and B}. The convexity of the numerical range, the bild and upper bild has been studied by several authors, see \cite{Ye1,Ye2,R,ST,STZ}.

In the complex setting the numerical range is convex thanks to the celebrated Toeplitz-Hausdorff Theorem \cite{GR}. Over the time, several generalizations of the numerical range have been proposed, namely the C-numerical range, the joint numerical range, among others, and in these cases convexity may fail. It then becomes natural to look for convexity-like geometric properties. For instance, the property of star-shapedness has been studied in \cite{CT,LLPS18,LLPS19,LNT,LP}. We recall that star-shapedness of a set $B$ only requires that there is an element $b_0\in B$ such that every segment  connecting $b_0$ and any other element of $B$ must be contained in $B$, see  definition \ref{def_ss and center}. Accordingly, we say that $b_0$ is in the star-center of $B$.

  For some generalizations of the numerical range, the  star-shapedness of the (complex) numerical range holds under certain conditions. In the article we tackle the question of the star-shapedness in the quaternionic setting. We prove that the quaternionic numerical range is always star-shaped. In addition, we characterize the shape of the star-center for quaternionic matrices.

The star-shapedness of the numerical range is a consequence of two simple facts (see theorem \ref{theo star shape}). Firstly, the convexity of the upper and lower bilds imply that the segments whose end is a real element of the bild is contained in the bild.  Therefore the bild is star-shaped and the reals therein are part of its center. And secondly, the equality, up to isomorphism, of all two dimensional real subalgebras of the quaternions that include the reals (as a real subspace), leads us to the conclusion that the reals in $W(A)$ are in fact part of the (star) center of the numerical range.

 As mentioned before, the general reason to focus on the bild is that the whole numerical range can be  reconstructed from it by using similarity classes. Our result is in line with the elements of the bild being the building blocks of the numerical range. In fact, we prove in theorem \ref{C(W)=[C(B)]} that the center of the numerical range is given by the similarity classes of the center of the bild. Therefore, we only need to know the center of the bild, and then to build the similarity classes to obtain the center of the numerical range. When the matrix is non hermitian the upper center (likewise for the lower center) is the region of the upper bild limited by two lines. These two lines are the tangents to the curve defining the boundary of the lower bild at the reals, see theorems \ref{does it_touch_your inner self?}, \ref{center bild} and corollary \ref{straightforward characterization}. As a consequence of these results we establish a new proof of the important theorem by Au-Yeung  \cite[theorem 3]{Ye1}, which establish a necessary and sufficient condition for convexity of the numerical range, see corollary \ref{Cor_AuYeung_result}. We conclude with an example where we explicitly compute the center.

\section{Preliminaries}
The quaternionic skew-field $\bH$ is an algebra of rank $4$ over $\R$ with basis $\{1, i, j, k\}$, where the product is given by $i^2=j^2=k^2=ijk=-1$. For any $q=a_0+a_1i+a_2j+a_3k\in\bH$ we denote by $q_r=a_0$ and $q_v=a_1i+a_2j+a_3k$, the real and imaginary parts of $q$, respectively. Let the pure quaternions be $\bP=\mathrm{span}_{\R}\,\{i,j,k\}$. The conjugate of $q$ is given by $q^*=q_r-q_v$ and the norm is defined by $|q|^2=qq^*$. Two quaternions $q,q'\in\bH$ are called similar, if there exists a unitary quaternion $s$ such that $s^{*}q' s=q$. Similarity is an equivalence relation and we denote by $[q]$ the equivalence class containing $q$. A necessary and sufficient condition for the similarity of $q$ and $q'$ is given by $q_r=q'_r \textrm{ and }|q_v|=|q'_v|$, see \cite[theorem 2.2.6]{R}.
We will denote the set of all equivalence classes of the elements of a set $X \subseteq \Hq$ by $[X]$. Then,
\[
[X]=\bigcup_{x \in X} [x]
\]

Let $\bH^n$ be the $n$-dimensional $\bH$-space. The norm of $\vc{x}\in \bH^n$ is $|\vc{x}|^2=\vc{x}^*\vc{x}$. The disk with center $\vc{a}\in\bH^n$ and radius $r>0$ is the set $\D_{\bH^n}(\vc{a},r)=\{\vc{x}\in\bH^n:|\vc{x}-\vc{a}|\leq r\}$ and its boundary is the sphere $\bS_{\bH^n}(\vc{a},r)$. In particular, if $\vc{a}=\vc{0}$ and $r=1$, we simply write $\D_{\bH^n}$ and $\bS_{\bH^n}$. With this notation, the group of unitary quaternions is $\bS_{\bH}$ whereas $\bS_{\bP}$ denotes the unit sphere over the pure quaternions.

Let $\M_n (\bH)$ be the set of all $n\times n$ matrices with entries over $\bH$. The set
\[W(A)=\{\vc{x}^*A\vc{x}:\vc{x}\in \bS_{\bH^n}\}\]
is called the quaternionic numerical range of $A$ in $\bH$. From the above definition we see that the quaternionic numerical range of $A \in \M_n(\Hq)$ is the subset of $\Hq$ containing the images of the quadratic function $f_A(\vc{x})=\vc{x}^*A\vc{x}$ over the quaternionic unitary sphere,  $\vc{x} \in \bS_{\Hq^n}$. The numerical range is invariant under unitary equivalence, i.e.
\[
W(U^*AU)=W(A)\,,
\]
for every unitary $U\in\mathcal{M}_n(\bH)$ \cite[theorem 3.5.4]{R}.

It is well known that if $q\in W(A)$ then $[q]\subseteq W(A)$, see \cite[page 38]{R}. This means that if $q_1\sim q_2$ and $q_2\in W(A)$ then $q_1\in W(A)$. For simplicity we just say that $q_2$ belongs to $W(A)$ by similarity. Therefore, it is enough to study the subset of complex elements in each similarity class. This set is known as $B(A)$, the bild of $A$:
\[
B(A)=W(A)\cap\C.
\]
We will freely use both notations $B(A)$ and $W(A)\cap\C$ for the bild of $A$. Although the bild may not be convex, the upper bild $B^+=W(A)\cap\C^+$ is always convex, see \cite{ST}. Analogously, the lower bild $B^-=W(A)\cap\C^-$ is also always convex. Note that $\C^+\cap \C^-=\R$, $B=B^+\cup B^-$ and $B^+\cap B^-\subseteq\R$.

For $p \in \bP$, let $\text{Span}\{1,p\}^+=\{\alpha+\beta p: \alpha \in \R, \beta \in \R_0^+\}$. For any $w \in W(A)$ and $p \in \bP$, let $w_{(p)}$ be the representative of the class $[w]$ in $\mathrm{span}\,\{1,p\}^+$, that is,
\[
\{w_{(p)}\}=[w]\cap\mathrm{span}\,\{1,p\}^+.
\]
In particular,
\[
\{w_{(i)}\}=[w]\cap\mathrm{span}\,\{1,i\}^+\subseteq B^+
\]
and we can write $w_{(i)}=w_r+i|w_v|$.

Let $V\subseteq\bH\cong\R^4$ be a real subspace of $\bH$. We denote by $\pi_V$ the canonical $\R$-linear projection $\pi_V:\bH\to V$.

For $h_0, h_1 \in \Hq$ we will denote by $[h_0,h_1]$ the set of convex linear combinations of $h_0$ and $h_1$:
\[
[h_0,h_1]=\{(1-\alpha)h_0+\alpha h_1: \alpha \in [0,1]\}.
\]
\begin{defn} \label{def_ss and center}
Let $B$ be a subset of a vector space. We say the set $B$ is star-shaped if there is a vector $b_0 \in B$ such that $[b_0,b]\subseteq B\,\,,\,\forall b\in B$. The star-center of a set $B$ is defined to be
\[
\cnt(B) =\{b_0 \in B: [b_0,b]\subseteq B, \text{ for any }  b \in B\}.
\]
\end{defn}
For simplicity, we refer to the star-center of a set as the center.

\section{Star-shapedness of the bild and numerical range}
The upper bild and the bild fully specify the numerical range, but the first is considered better suited to represent the quaternionic numerical range. This is not only because it is convex but also because it has the advantage of containing one single element from each similarity class. In a sense, the upper bild can be interpreted as the set of equivalence classes for the similarity relation $\sim$, that is, the quotient set $B^+=W/\sim$. However, from the convexity of the upper bild we cannot infer about the convexity of the numerical range, as the first is always convex and the latter is not.

The first result of this paper relates the convexity of the bild with the convexity of the numerical range. This is a known result (see \cite[page 53]{Zh}), however we present a different proof based on elementary properties of the numerical range.

\begin{thm}\label{theo convex W and B}
Let $A\in  \M_n (\Hq)$. Then $W(A) \cap \C$ is convex if and only if  $W(A)$ is convex.
\end{thm}
\begin{proof}
It is enough to prove that, if $W(A) \cap \C$ is convex then $W(A)$ is convex.
Let $a, b \in W(A)$ and $\alpha \in [0,1]$. We need to show that $c=\alpha a+(1-\alpha) b \in W(A)$. The quaternion $c=c_r+c_v$ has  $c_r=\alpha a_r+(1-\alpha)b_r$ and
\begin{equation}\label{modulo c_v}
|c_v|=|\alpha a_v+(1-\alpha)b_v|\leq \alpha |a_v|+(1-\alpha)|b_v|.
\end{equation}
We will prove that $c_{(i)} \in B^+$, thus proving by similarity that $c\in W(A)$. Since the upper bild is convex,
\begin{equation}\label{omega in B^+}
\omega=\alpha a_{(i)} +(1-\alpha) b_{(i)} = c_{r} + \big(\alpha|a_v|+(1-\alpha)|b_v|\big)i  \in  B^+.
\end{equation}

\smallskip

By similarity, $\omega^* \in  B^-$. Note that $c_{(i)}=c_{(i),r}+c_{(i),v}=c_r+i|c_v|$. From (\ref{omega in B^+}), $c_{(i),r}=\omega_r=\omega^*_r$ and from (\ref{modulo c_v}), $ | c_{(i),v}|\leq |\omega_v| $. Therefore,
\[
-\frac{\omega_v}{i}\leq\frac{c_{(i),v}}{i}\leq\frac{\omega_v}{i}\,,
\]
and so there is $\beta \in [0,1]$ such that $c_{(i),v}=\beta\omega_v+(1-\beta)\omega_v^*$. Hence, $c_{(i)}=\beta\omega+(1-\beta)\omega^*$. By hypothesis, $W(A) \cap \C $ is convex and so $c_{(i)} \in W(A) \cap \C$.
\end{proof}
Any quaternionic matrix $A \in \M_n (\Hq)$ can be written as $A=\tilde H+\tilde S$, with $\tilde H =\frac{A+A^*}{2}$ hermitian and $\tilde S =\frac{A-A^*}{2}$ skew-hermitian. Let $U \in \M_n(\Hq)$ be the unitary matrix that diagonalize $\tilde S$, i.e., $S=U^*\tilde S U=\mathrm{diag}\,(s_1,\ldots,s_n)$. Since the numerical range is invariant under unitary equivalence, we can work with $U^*AU$, that can be written in the form $U^*AU= U^*\tilde H U+U^*\tilde S U=H+S$. Since $H$ is hermitian $f_{H}(\vc{x}) \in \bR$ and since $S$ is skew-hermitian the real part of $f_{S}(\vc{x})$  is zero, see \cite[corollary 3.5.3]{R}.

We claim that $0 \in W_{\Hq}(S)$. To prove this we will find a vector $\vc{x} \in \bS_{\Hq^n}$ such that $f_{S}(\vc{x})=0$. Let $x_3=\ldots=x_n=0$, then take $z_1$ and $z_2$ in $\bS_{\Hq}$ such that $q_1=z_1^* s_{1} z_1 \in \C^+$ and $q_2=z_2^*  s_{2} z_2 \in \C^-$. The quaternions $q_1$ and $q_2$ are either zero or the representatives of $s_1$ in $\C^+$ and $s_2$ in $\C^-$, respectively. Thus they are pure complex. Finally, choose $\beta \in [0,1]$ such that $\beta q_1+(1-\beta)q_2=0$. Take $x_1=\beta^{1/2} z_1$ and $x_2=(1-\beta)^{1/2} z_2$. Then, the vector $\vc{x} \in \bS_{\Hq^n}$ is in the stated conditions. It is now clear that $W(A) \cap \bR\neq \emptyset$. In fact, take vector $\vc{x}$ and compute $f_A(\vc{x})=f_{H}(\vc{x})+f_{S}(\vc{x})=f_{H}(\vc{x}) \in \bR$. We have proved the following result. \footnote{This result is apparently known for some time, as it appears in the thesis of \cite{Siu}, supervised by Au-Yeung, however it has never been published before, (to the best of our knowledge). In spite of this, Au-Yeung in \cite[corollary 1]{Ye1}  apropos of the connectedness of $W_{\Hq} \cap \bR $, and citing a result from \cite{J}, states the possibility of $W_{\Hq} \cap \bR =\emptyset$. This possibility is also stated by \cite[theorem 9.2]{Zh} and \cite[corollary 2.10]{K}, repeating again the same result by \cite{J}, (although \cite{K} doesn't cite it).}

\begin{prop}\label{n.r. has reals}
For any $A \in \M_n(\Hq)$, $W(A) \cap \bR\neq \emptyset$.
\end{prop}

From now on, we fix a matrix with quaternionic entries, $A \in \M_n(\Hq)$, and we denote the quaternionic numerical range of $A$ simply by $W=W(A)$.

Let $q_1, q_2 \in \bS_\bP$.  We say an element $a_1 \in \mathrm{span}\{1, q_1\}$ is $\dot\sim$-similar to  $a_2 \in \mathrm{span}\{1, q_2\}$, if and only if, for some $r, s \in \R$,
\begin{equation*}
a_1=r+s q_1 \text{ and } a_2=r+s q_2,
\end{equation*}
in which case we write $a_1\dot\sim a_2.$ We say that $A_1\subseteq \mathrm{span}\{1, q_1\}$ and $A_2\subseteq \mathrm{span}\{1, q_2\}$ are $\dot\sim$-similar, and denote it by $A_1 \dot\sim A_2$, if and only if, for any $a_1 \in A_1$ there is an $a_2 \in A_2$ such that $a_1 \dot\sim a_2$, and vice versa. When two sets are $\dot\sim$-similar they share some properties, namely convexity. In fact, if $A_1$ is convex we can conclude that $A_2$ is convex. Take any $a_2,\tilde a_2\in A_2$. Then, there are $a_1, \tilde a_1 \in A_1$, such that $a_1\dot\sim a_2$ and $\tilde a_1\dot\sim \tilde a_2$. For any $\alpha \in [0,1]$ it is a matter of simple calculations to note that
\[
\alpha a_1+ (1-\alpha)\tilde a_1 \dot\sim \alpha a_2+ (1-\alpha)\tilde a_2.
\]

Now, since $A_1 \dot\sim A_2$ and $A_1$ is convex  we conclude that $\alpha a_2+ (1-\alpha)\tilde a_2 \in A_2$. Therefore $A_2$ is also convex. A similar argument proves that the centers are $\dot\sim$-similar for any two $\dot\sim$-similar sets $A_1$ and $A_2$, since whenever a segment is in $A_1$ the $\dot\sim$-similar segment must be in $A_2$. That is, $\cnt(A_1) \dot\sim \cnt(A_2)$ whenever $A_1 \dot\sim A_2$.

Define, for $q\in\bS_{\bP}$, $W^{(q)}=W\cap \mathrm{span}\{1,q\},$  $W^{(q)+}=W \cap \mathrm{span}\{1,q\}^+$ and  $W^{(q)-}=W \cap \mathrm{span}\{1,q\}^-.$

\begin{lemma}\label{lemma_wq+-convex}
For any $q_1, q_2 \in \bS_\bP$ we have:
\begin{enumerate}[(i)]
  \item $W^{(q_1)+}, W^{(q_1)-}$ are convex,
  \item $\cnt\Big(W^{(q_1)}\Big)\dot\sim \cnt\Big(W^{(q_2)}\Big).$
\end{enumerate}
\end{lemma}

\begin{proof}
The numerical range is such that, by similarity, $W^{(q_1)}\dot\sim W^{(q_2)}$, for any $q_1,q_2 \in \bS_\bP$. It is also an immediate conclusion of numerical range's closedness to similarity that $W^{(q_1)+}\dot\sim W^{(q_2)+}$, for any $q_1,q_2 \in \bS_\bP$. It is known that the upper bild $W^{(i)+}=B^+$ is convex, thus from the previous discussion, we have that $W^{(q)+}$ is also convex for any $q \in\bS_\bP$. Moreover from  $W^{(q_1)}\dot\sim W^{(q_2)}$ we know that $\cnt\Big(W^{(q_1)}\Big)\dot\sim \cnt\Big(W^{(q_2)}\Big)$. \end{proof}

As a consequence of this lemma we only need to study the center of one of the $W^{(q)}$'s and the natural choice is to take $q=i$, that is, we only need to study the center of the bild $B=W^{(i)}$.

\begin{thm}\label{theo star shape}
The quaternionic numerical range $W$ is star-shaped and $W\cap \bR \subseteq \cnt(W)$.
\end{thm}
\begin{proof}
By proposition \ref{n.r. has reals}, there is $r\in W(A)\cap\R$. For every $\omega\in W$, there is $q\in\bS_{\bP}$ such that $\omega\in W^{(q)}$. Since $W^{(q)}=W^{(q)+}\cup W^{(q)-}$ and using lemma \ref{lemma_wq+-convex} we have that $[r,\omega]\subseteq W^{(q)}\subseteq W$. Hence, the numerical range is star-shaped. Moreover, $W\cap\R\subseteq\cnt(W)$.
\end{proof}
The numerical range $W(A)$ is contained in $\R$ if, and only if, $A$ is hermitian, see \cite[corollary 3.5.3]{R}.. The next result follows trivially from theorem \ref{theo star shape}.
\begin{cor}\label{center_hermitian}
If $A$ is hermitian then $\cnt(W(A))=W(A)$.
\end{cor}

\begin{lemma} \label{center_bild_conj}
The center of the bild is closed under conjugation, i.e.
\[
\cnt(W\cap \C) =\cnt(W\cap \C)^*.
\]
\end{lemma}
\begin{proof}
Assume $c \in \cnt(W\cap \C)$. Let $\omega $ be any element of the bild $\omega \in W \cap \C$. Since the bild is closed for conjugation, $\omega^* \in W \cap \C$. Then $c$ being in the center implies that $\alpha c+(1-\alpha)\omega^* \in W \cap \C$, for any $\alpha\in [0,1]$. And again using the bild's closedness to conjugation we conclude that $\alpha c^*+(1-\alpha)\omega \in W \cap \C$. Since this is true for any $\omega\in W \cap \C$, $c^* \in  \cnt(W\cap \C)$. The converse inclusion follows similar steps.
\end{proof}

We now establish the equality between the center of the bild and the complex part of the center of the numerical range.

\begin{prop}\label{bilds_center}
We have:
\[\cnt(W) \cap \C = \cnt(W \cap \C).\]
\end{prop}

\begin{proof}
The inclusion $\cnt(W) \cap \C \subseteq \cnt(W \cap \C)$ is obvious since a complex element in the center of $W$ must be in the center of $W\cap \C$.

For the converse inclusion, starting with $c \in \cnt(W \cap \C)$, we will prove that $y=\alpha c+(1-\alpha) \omega\in W$ for any $\alpha \in [0,1]$ and $\omega \in W$.

We can assume, without loss of generality, that $ c=c_{(i)} \in \C^+$. Since any quaternion $y$ can be written as the sum of a real with a pure quaternion, we may write $y=y_r+|y_v|q$, with $q\in\bS_{\bP}$. We have:
\[
y=\Big(\alpha c_r+(1-\alpha)w_r \Big)+ \Big|\alpha c_v+(1-\alpha)w_v \Big|q.
\]
By similarity, it is enough to prove that $y_{(i)} \in B^+$. With this purpose, we will find two elements $a, b\in B^+$ such that
\begin{equation}\label{real part equal}
a_r=b_r=y_r\,\,\,\textrm{and}\,\,\,|a_v|\leq|y_v|\leq|b_v|.
\end{equation}

In this case, by convexity of the upper bild, $y_{(i)} \in B^+$ since $y_{(i)}= \beta a + (1-\beta)b,$
for some $\beta \in [0,1]$.
Let
\begin{align*}
b & =\alpha c_{(i)} +(1-\alpha) w_{(i)} \\
& = y_r+\Big(\alpha|c_v|+(1-\alpha)|w_v|\Big)i\in W\cap\C^+
\end{align*}

The conclusion that $ b \in B^+$ follows from the fact that $w_{(i)}, c_{(i)}\in B^+$, which is a convex set.

If $\alpha |c_v|-(1-\alpha)|w_v| >0 $ we take $a=\alpha c_{(i)} +(1-\alpha) w_{(i)}^*$, else we take $a=\alpha c_{(i)}^* +(1-\alpha) w_{(i)}$ (clearly, $a\in\C^+$).

We now need to check that $a$ and $b$ are in $W$ and satisfy conditions (\ref{real part equal}). It is trivial to conclude that the real parts are all equal. On the other hand,
\[
|b_v|=\alpha |c_v|+(1-\alpha)|w_v|\geq |\alpha c_v+(1-\alpha)w_v|=|y_v|.
\]
To conclude that $|a_v| \leq |y_v|$ we will use Cauchy-Schwartz inequality. If we look a quaternion $q\in \Hq$ as a vector in $\R^4$, its norm is given by  $\langle q, q\rangle=|q|^2$, where
 $\langle.,.\rangle$ is the usual inner product in real vector spaces. Then we have:
\begin{align*}
|y_v|^2&= \Big\langle \alpha c_v +(1-\alpha) w_v, \alpha c_v +(1-\alpha) w_v \Big\rangle\\
& =\alpha^2 |c_v|^2 +(1-\alpha)^2 |w_v|^2+\alpha(1-\alpha)\Big(\langle c_v, w_v \rangle+ \langle w_v, c_v \rangle\Big )\\
& \geq \alpha^2 |c_v|^2 +(1-\alpha)^2 |w_v|^2-2\alpha(1-\alpha)|c_v| |w_v|\\
& = \big(\alpha |c_{v}| -(1-\alpha) |w_{v}|\big)^2.
\end{align*}

Since $|c_v|=|c_{(i),v}|$ and$|w_v|=|w_{(i),v}|$, we have:
\[
|y_v|^2\geq \big(\alpha |c_{(i),v}| -(1-\alpha) |w_{(i),v}|\big)^2.
\]
Using the equality
$(\alpha c_{(i)} +(1-\alpha) w_{(i)}^*\Big)_v=\alpha |c_{(i),v}|i -(1-\alpha) |w_{(i),v}|i$,
it follows that
\begin{align*}
|y_v|^2 & \geq \Big|\Big( \alpha c_{(i)} +(1-\alpha) w_{(i)}^*\Big)_v\Big|^2\\
& = \Big|\Big( \alpha c_{(i)}^* +(1-\alpha) w_{(i)}\Big)_v\Big|^2\\
& = |a_v|^2 .
\end{align*}
Therefore,
$|a_v|\leq|y_v|\leq|b_v|.$

It remains to prove that $a \in W$. If $a=\alpha c_{(i)} +(1-\alpha) w_{(i)}^*$, by hypothesis $c_{(i)}\in \cnt(W\cap \C)$ and $w_{(i)}^* \in W\cap \C$, then any convex combination of them is also in $W \cap \C$. If $a=\alpha c_{(i)}^* +(1-\alpha)w_{(i)}$ then $a \in W$, because $c_{(i)}^* \in \cnt(W \cap \C)$ by lemma \ref{center_bild_conj}, and $w_{(i)} \in W \cap \C$.
\end{proof}
Next result establish the relation between the center of the numerical range $\cnt(W)$ and the center of the bild $\cnt(W\cap \C)$.
 \begin{thm}\label{C(W)=[C(B)]}
The center of the numerical range is such that
\[
\cnt\big(W\big)=\Big[\cnt(W \cap \C) \Big].
\]
\end{thm}
\begin{proof}
Let $c\in \cnt(W)$. For some $q\in \bS_{\bP}$, we have $c\in \cnt(W)\cap\mathrm{span}\,\{1,q\}$. Using a similar reasoning of the proof of proposition \ref{bilds_center}, we can show that
\[
\cnt(W)\cap\mathrm{span}\,\{1,q\}=\cnt(W\cap\mathrm{span}\,\{1,q\})=\cnt(W^{(q)}).
\]
Now, $c\in\cnt(W)$ if and only if $c\in\cnt(W^{(q)})$, for some $q\in\bS_{\bP}$, that is,
\[
c\in\cnt(W) \Leftrightarrow c\in\cnt(W^{(q)}), \text{for some}\, q\in \bS_{\bP}.
\]
By lemma \ref{lemma_wq+-convex}, $\cnt(W^{(q)})\dot\sim \cnt(W^{(i)})$.
We conclude that
 \[
  c\in[\cnt(W^{(q)})]=[\cnt(W^{(i)})]=[\cnt(W\cap\C)].
  \]\end{proof}

 If we use the fact that $W$ is the set of all elements similar to those in $W \cap \C$, that is, $W=\Big[W \cap \C\Big]$, the above result can be written in the following way:
\[
\cnt\Big(\big[W \cap \C\big]\Big)=\Big[\cnt(W \cap \C) \Big].
\]

In other words, the operations of taking the center and of taking the equivalence classes of a numerical range commute.

\section{Characterization of the center of the bild}

We now know that it is possible to characterize the center of the numerical range from the center of the bild. On the other hand, lemma \ref{center_bild_conj} guarantees that the lower part of the center of the bild is the conjugate of the upper part,
\begin{equation}\label{center+center-}
  \cnt^-=\cnt(W) \cap \C^-=(\cnt^+)^*=(\cnt(W) \cap \C^+)^*,
\end{equation}
and we conclude that to determine $\cnt(W)$ we only need to know $\cnt^+$. From corollary \ref{center_hermitian}, we may focus only on non-hermitian matrices.

By the convexity of the upper bild, the segment joining any two elements in the upper bild is contained in it. Therefore an element of the upper bild is not in the center if and only if a convex combination with an element in the lower bild is not in the bild.  That is, an element $\vc{\omega} \in W \cap \C^+$ is not in the center of the bild, $\vc{\omega} \not\in \cnt(W \cap \C)$, if and only if, there is $\vc{z} \in W \cap \C^-$ such that the segment  connecting the two is not contained in the bild, i.e. $[\vc{\omega}, \vc{z}] \not\subseteq W \cap \C$.
The argument we will use is build upon the fact that a segment, joining two elements of the bild, is not totally contained in the bild, if and only if it crosses the reals outside of it. Thus, either an element $\vc{\omega}$ of the upper bild has all its segments $[\vc{\omega}, \vc{z}]$, for $\vc{z}\in W \cap \C^-$, crossing the real line inside the bild, that is, $[\vc{\omega}, \vc{z}]\cap \R \subseteq B$, in which case $\vc{\omega}$ is in the center, or there is one of these segments that crosses the real line outside the bild, and the element $\vc{\omega}$ is not in the center.

For the rest of this section we will slightly change notation and write $z=x+iy$ as $(x,y)$. Let $m= \min W \cap \R$ and $M= \max W \cap \R$ be the minimum and maximum of the real elements in the bild. Using the previous reasoning, but on a dual perspective, to find out if an element $\vc{\omega}$ in the upper bild is in the center, we only need to see if the segments joining  $\vc{\omega}$  to $(M,0)$ and to $(m,0)$ intersects the interior of the lower bild or not. In the case where it does the element is not in the center. For instance, if the segment joining $\vc{\omega} \in B^+$  to $(m,0)$ intersects $B^-$ at $\vc{z}$ in the lower part of the interior of the bild, then there is an element $\vc{\tilde{z}}$ to the left of $ \vc{z}$ such that the segment  $[\vc{\omega}, \vc{\tilde{z}}]$ will cross the reals to the left of  $(m,0)$, and therefore outside of the bild.

The next results formalize this intuitive argument. To reach this we will need to define for each $\vc{\omega} \in \C^+$ two lines, one denoted $l_{\vc{\omega}}$ connecting $\vc{\omega}=(\omega_1,\omega_2)$ to $(m,0)$, and the other denoted $L_{\vc{\omega}}$ connecting $\vc{\omega}$ to $(M,0)$. Since the real points of the numerical range belongs to the center (see theorem \ref{theo star shape}), it is enough to consider points $\vc{\omega}=(\omega_1, \omega_2)$, with $\omega_2>0$. The lines are given by
\begin{align*}
&l_{\vc{\omega}}=\{ (x,y): x=ay+m \} \\
&L_{\vc{\omega}}=\{ (x,y): x=by+M) \},
\end{align*}
with $a=\dfrac{\omega_1-m}{\omega_2}$ and $b=\dfrac{\omega_1-M}{\omega_2}$.


Let $y_m=\min \{ \pi_{\text{Span}\{i\}}(B)\}$ and $y_M=\max \{ \pi_{\text{Span}\{i\}}(B)\}$. By symmetry of the bild, $y_M=-y_m$. Since the matrix is non-hermitian, $y_M>0$.

We may define, for $y\in[y_m,0]$, two functions:
\[
x_1(y)=\mathrm{min}\,\{x:(x,y)\in B^-\} \,\,\,\textrm{and}\,\,\,x_2(y)=\mathrm{max}\,\{x:(x,y)\in B^-\}.
\]
Notice that $x_1(0)=m$ and $x_2(0)=M$. According to \cite[theorem 5.3]{Roc}, $x_1(\cdot)$ is convex and $x_2(\cdot)$ is concave. The lower bild may be written using $x_1(\cdot)$ and $x_2(\cdot)$:

\smallskip

\begin{equation}\label{lower bild}
B^-=\big\{(x,y): y_m\leq y \leq 0 \text{ and } x_1(y) \leq x \leq x_2(y) \big\}.
\end{equation}

\medskip

The interior of the lower bild is given by:

\smallskip

\begin{equation}\label{lower part interior}
\interior{(B^-)}=\big\{(x,y): y_m< y < 0 \text{ and } x_1(y) < x < x_2(y) \big\}.
\end{equation}

\medskip

The next result gives a characterization of $\cnt(B)$, when $m<M$.
  For $\vc\omega \in B^+$  the lines $l_{\vc\omega}$ and $L_{\vc\omega}$ do not cross over the interior of the lower bild, if and only if, $\vc\omega\in \cnt(B)$.

\begin{thm} \label{does it_touch_your inner self?}
Let  $m<M$ and let $\vc{\omega} \in B^+$. Then,
$\vc{\omega} \in \cnt(B)$
if, and only if,
\[
\quad \Big(l_{\vc{\omega}}\cap \interior{(B^-)}\Big) \bigcup \Big(L_{\vc{\omega}}\cap \interior{(B^-)}\Big)= \emptyset.
\]
\end{thm}

\begin{proof}
We begin by observing the following. Let $\vc{\omega}=(\omega_1,\omega_2)\in B^+$ with $\omega_2> 0$. The line $l_{\vc{\omega}}$ passing through $\vc{\omega}$ and $(m,0)$ can be written as:
\[
l_{\vc{\omega}}(y)=\frac{\omega_1-m}{\omega_2}(y-\omega_2)+\omega_1
\]
and define two half planes:
\[
\wp^-: x-l_{\vc{\omega}}(y)<0 \quad \text{and} \quad \wp^+: x-l_{\vc{\omega}}(y)>0.
\]

To prove that if $\vc{\omega} \in \cnt(B)$, then
$(l_{\vc{\omega}}\cup L_{\vc{\omega}})\cap \interior{(B^-)}= \emptyset$
we proceed by contrapositive.

Fix an element $\vc{\omega}=(\omega_1, \omega_2)\in B^+$ as before, i.e., with $\omega_2\neq 0$, and suppose there is an element $\vc{z}\in l_{\vc{\omega}}\cap \interior{(B^-)}$ (if $\vc{z}\in L_{\vc{\omega}}\cap \interior{(B^-)}$ the proof is analogous).
Since $\vc{z}=(z_1, z_2)\in l_{\vc{\omega}}$, the line $l_{\vc{\omega}}$ may also be written as
\[
f(y)=\frac{\omega_1-z_1}{\omega_2-z_2}(y-\omega_2)+\omega_1.
\]

Let $N_{\varepsilon}(\vc{z})\subset \interior{(B^-)}$ be a neighborhood of $\vc{z}$. Then, there is $\vc{\tilde{z}}=({\tilde{z}}_1, {\tilde{z}}_2)\in N_{\varepsilon}(\vc{z})$ such that $\tilde{z_1}<z_1$ and $\tilde{z_2}=z_2$.

The line $\tilde{l}$ passing through $\vc{\omega}$ and $\vc{\tilde{z}}$ is
\[
g(y)=\frac{\omega_1-\tilde{z}_1}{\omega_2-z_2}(y-\omega_2)+\omega_1.
\]

Define the affine function $h(.)$ by:
\[
h(y)\equiv g(y)- f(y)=\frac{z_1-\tilde{z}_1}{\omega_2-z_2}(y-\omega_2).
\]

Clearly, $h(\omega_2)=0$ and $h(z_2)=\tilde{z}_1-z_1<0$. Since $\omega_2>0$ and $z_2<0$, there is $\beta \in (0,1)$ such that $0=\beta\omega_2+(1-\beta)z_2$. Moreover, since $h$ is affine,
\begin{eqnarray*}
  h(0) &=& \beta h(\omega_2)+(1-\beta)h(z_2) \\
   &=& (1-\beta)h(z_2)<0
\end{eqnarray*}
and so, $g(0)<f(0)=l_{\vc{\omega}}(0)=m$. Hence, the line passing through $\vc{\omega}$ and $\tilde{\vc{z}}$ does not intersect $B\cap\R$,  which implies that $[\vc{w}, \tilde{\vc{z}}]\nsubseteq B$ and $\vc{\omega}\notin \cnt(B)$.

Now we prove the converse, that is, for $\vc{\omega}\in B^+$ if $\vc{\omega}\notin \cnt(B)$ then $(l_{\vc{\omega}}\cup L_{\vc{\omega}})\cap \interior{(B^-)}\neq \emptyset.$ Since $\vc{\omega}\notin \cnt(B)$, there is a point $\vc{z}=(z_1, z_2)\in {B}^-$ such that the line segment $[\vc{\omega}, \vc{z}]$ is not contained in the bild.

Assume that the line containing $[\vc{\omega}, \vc{z}]$, call it $x=g(y)$, intersects the real line ($y=0$) at $(\nu, 0)$. Since $[\vc{\omega},\vc{z}]\nsubseteq B$ then $\nu\notin [m, M]$. Otherwise, if $(\nu,0)\in B$, convexity of the upper bild implies that $[\vc{\omega},(\nu,0)]\subseteq B^+$ and convexity of the lower bild implies that $[(\nu,0),\vc{z}]\subseteq B^-$. Thus,
\[
[\vc{\omega},\vc{z}]=[\vc{\omega},(\nu,0)]\cup[(\nu,0),\vc{z}]\subseteq B,
\]
which contradicts our hypothesis. We will assume that $\nu<m$ (when $\nu>M$ the proof is analogous). We claim that we can take $\vc{z}=(z_1,z_2)\in \interior{(B^-)}$. In fact, if $(z_1,z_2)$ is on the boundary of $B$, take $\vc{z}_{\vc{\epsilon}}=(z_1+\epsilon_1,z_2+\epsilon_2)$, with $\vc{\epsilon}=(\epsilon_1,\epsilon_2)$ small enough such that $\vc{z}_{\vc{\epsilon}}\in\interior{(B^-)}$ and $[\vc{\omega}, \vc{z}_{\vc{\epsilon}}]\cap \R =\nu'$, close enough to $\nu$ in order to satisfy $\nu'<m$. In this way, there is a point $\vc{z}_{\vc{\epsilon}}\in\interior{(B^-)}$ such that $[\vc{\omega}, \vc{z}_{\vc{\epsilon}}]\nsubseteq B$.

Since $\nu <l_{\vc{\omega}}(0)=m<M$ then $(\nu, 0)$ and $(M, 0)$  must be in different half-planes, that is, $(\nu, 0)\in \wp^-$ and $(M, 0)\in \wp^+$.

We now show that $\vc{z}$ and $(M, 0)$ are in different half-planes using the same reasoning of the first part of the proof, but now with $g$ being the line that passes through the points $\vc{\omega}$, $(\nu,0)$ and $\vc{z}$, and $f$ being the line that contains $[\vc{\omega},(m,0)]$. It follows that $h(0)=g(0)-f(0)=\nu-m<0$. Since $\omega_2>0$ and $z_2<0$, there is $\beta \in (0,1)$ such that $0=\beta\omega_2+(1-\beta)z_2$. Hence, $h(0)=(1-\beta)h(z_2)<0$ and $h(z_2)<0$. It follows that $z_1<l_{\vc{\omega}}(z_2)$ and therefore $\vc{z}\in \wp^-.$

Let $\gamma(x,y)=x-l_{\vc{\omega}}(y)$. Since $\gamma(\vc{z})<0$ and $\gamma(M,0)>0$, then the line that joins $\vc{z}$ to $(M,0)$, by continuity of $\gamma$, passes through a point $\vc{z}'$ with $\gamma(\vc{z}')=0$, that is, $[\vc{z}, (M, 0)]\cap l_{\vc{\omega}}=\vc{z}'$. Taking into account $\vc{z}'\in l_{\vc{\omega}}$, it only remains to prove that $\vc{z}'\in \interior{(B^-)}$. Since $\vc{z}=(z_1,z_2), (M,0)\in B^-$, by convexity of $B^-$ we have
\begin{equation*}\label{z' convex comb}
\vc{z'}=(z'_1,z'_2)=(1-\alpha)(z_1,z_2)+\alpha (M,0)\,,
\end{equation*}
for some $\alpha\in(0,1)$. Note that if $\alpha=0$, $\vc{z}'=\vc{z}$ and if $\alpha=1$, $\vc{z}'=(M,0)$, that cannot happen because $\vc{z}'\in l_{\vc{\omega}}$ and $\vc{z}, (M,0)\notin l_{\vc{\omega}}$. We know that $\vc{z}\in \interior{(B^-)}$, and so,
\begin{equation}\label{x_1<z_1<x_2}
x_1(z_2)<z_1<x_2(z_2).
\end{equation}

From (\ref{x_1<z_1<x_2}) and since $M> x_1(0)=m$, we have:
\begin{eqnarray*}
  z'_1 &=& (1-\alpha) z_1+ \alpha M> (1-\alpha) x_1(z_2)+\alpha x_1(0) \\
   &\geq& x_1((1-\alpha)z_2+\alpha 0)\,\,\,\,\,\,\,\,\,\,\,\,\,\,\,\textrm{(by convexity of }\,x_1(\cdot)\,\textrm{)}\\
   & = & x_1(z'_2).
\end{eqnarray*}

With a similar reasoning and using that $\alpha\in (0,1)$ we see that $z'_1<x_2(z'_2)$. It follows that
\[
x_1(z'_2)<z'_1<x_2(z'_2).
\]
Now we need to check that $y_m<z'_2<0$. Since $\vc{z}\in \interior{(B^-)}$ we know that $y_m<z_2<0$, and so,
\[
y_m<(1-\alpha) y_m<(1-\alpha) z_2=z'_2<0.
\]
We conclude that $\vc{z}'\in\interior{(B^-)}$.
\end{proof}

Relying on our previous results, we will now prove the existence of two lines containing $(m,0)$ and $(M,0)$ that define the upper boundary of the center. Such lines are denoted respectively by $l$ and $L$.

Any concave function has lateral derivatives \cite[theorem 23.1]{Roc}, therefore let $a=x'_1(0^-)$ and $b=x'_2(0^-)$, the left derivative at $0$ of $x_1(\cdot)$ and $x_2(\cdot)$, respectively.

Let the left tangent line to $x_1$ and $x_2$ at $0$ be given by the sets
\begin{equation}\label{tangent lines}
 \{(x,y): x=l(y)=ay+m\}\quad \text{and}\quad
\{(x,y): x=L(y)=by+M\},
\end{equation}
respectively. Since $x_1(\cdot)$ is convex and $x_2(\cdot)$ is concave we have, \cite[theorem 25.1]{Roc}, $l(y) \leq x_1(y)$ and $x_2(y) \leq L(y)$, for every $y\in [y_m, 0]$.

\begin{prop} \label{prop retas}
 Let $m<M$ and let $\vc{\omega}=(\omega_1,\omega_2) \in B^+$. Then,
 \begin{enumerate}[(i)]
   \item $l(\omega_2)\leq \omega_1$ if, and only if, $l_{\vc{\omega}}\cap \interior{(B^-)}=\emptyset$,
   \item $\omega_1\leq L(\omega_2)$ if, and only if, $L_{\vc{\omega}}\cap \interior{(B^-)}=\emptyset$.
 \end{enumerate}
\end{prop}
\begin{proof}
We will prove (i). A similar reasoning proves (ii). Let $\vc{\omega}\in B^+$ with $\omega_2>0$ and $l_{\vc{\omega}}$ be the line passing through $\vc{\omega}$ and $(m,0)$. We can write $l_{\vc{\omega}}(y)=\tilde{a}y+m$, with $\tilde{a}=\dfrac{\omega_1-m}{\omega_2}$.

Now we will prove that if $\omega_1\geq l(\omega_2)$ the line $l_{\vc{\omega}}$ does not intersect $\interior{(B^-)}$.
Since $\omega_1=l_{\vc{\omega}}(\omega_2)\geq l(\omega_2)$, it is clear that $\tilde{a}\omega_2+m \geq a\omega_2+m$, i.e., $(\tilde{a}-a)\omega_2\geq 0$. Since $\omega_2> 0$ we have $\tilde{a} \geq a$. For $y\geq 0$, $(l_{\vc{\omega}}(y), y)\in \C^+$ and so
\[
(l_{\vc{\omega}}(y), y)\notin \C^- \supseteq \interior{(B^-)}.
\]
For $y<0$ we have $l_{\vc{\omega}}(y)\leq l(y)$, since $\tilde{a}\geq a$. From the convexity of $x_1(\cdot)$ and using \cite[theorem 25.1]{Roc} we have
$x_1(y)\geq l(y)$ for any $y\in [y_m,0]$. Then,
$
l_{\vc{\omega}}(y)\leq x_1(y).
$
Therefore, $(l_{\vc{\omega}}(y), y)\in l_{\vc{\omega}}$ with $l_{\vc{\omega}}(y)\leq x_1(y)$ and from (\ref{lower part interior}) we see that $(l_{\vc{\omega}}(y), y)\notin \interior{(B^-)}$.

To prove the converse, we want to show that if $l(\omega_2)>\omega_1$ then $l_{\vc{\omega}} \cap \interior{(B^-)}\neq \emptyset$, that is, the line $l_{\vc{\omega}}\supseteq [(\omega_1,\omega_2), (m,0)]=[\vc{\omega}, (m,0)]$ intersects $\interior{(B^-)}$. Again, we have $\omega_1=l_{\vc{\omega}}(\omega_2)< l(\omega_2)$, and therefore $(\tilde{a}-a)\omega_2< 0$. Since $\omega_2> 0$, necessarily $\tilde{a} < a$. Define, for $y\in [y_m, 0]$,
\[
h(y)= l_{\vc{\omega}}(y)-x_1(y).
\]
By the first order Taylor's approximation of $x_1(\cdot)$, for small $\epsilon>0$ we get
\begin{eqnarray*}
h(-\epsilon)   &=& l'_{\vc{\omega}}(0^-)(-\epsilon)-x_1'(0^-)(-\epsilon)+ o(-\epsilon)\\
   &=&  \tilde{a} (-\epsilon) -a (-\epsilon) + o(-\epsilon) \\
   &=&  -\epsilon\Big(\tilde a - a + \frac{o(-\epsilon)}{-\epsilon}\Big).
\end{eqnarray*}
Since $\tilde{a}<a$, it follows that $h(-\epsilon)>0$, for small $\epsilon >0$. In other words, $l_{\vc{\omega}}(-\epsilon)>x_1(-\epsilon)$ for $\epsilon$ small enough. Taking into account that $l_{\vc{\omega}}(0)=m<x_2(0)=M$ and that $l_{\vc{\omega}}$ and $x_2(\cdot)$ are continuous  \cite[corollary 10.1.1]{Roc}, for $\epsilon$ small enough we have
$l_{\vc{\omega}}(-\epsilon)< x_2(-\epsilon)$. Therefore, we can choose an $\epsilon>0$ such that $x_1(-\epsilon)<l_{\vc{\omega}}(-\epsilon)<x_2(-\epsilon)$ and $y_m<-\epsilon <0$. Then $(l_{\vc{\omega}}(-\epsilon),-\epsilon)\in\interior{(B^-)}$. \end{proof}

We can now present a general way to determine the center. Let $\pi_m \equiv\mathrm{min}\,\pi_{\R}(W)$ and $\pi_M \equiv\mathrm{max}\,\pi_{\R}(W)$.

\begin{thm}\label{center bild}
 Let $\vc{\omega}=(\omega_1,\omega_2) \in B^+$. Then, $l(\omega_2)\leq \omega_1\leq L(\omega_2)$ if, and only if, $\vc{\omega}\in \cnt(B)$.
\end{thm}

\begin{proof}
When $m<M$, proposition \ref{prop retas}  and  theorem \ref{does it_touch_your inner self?} prove the stated equivalence. For the case $m=M$ we will first find out the $\cnt(B)$ and then prove the equality with the set $\{(\omega_1,\omega_2) \in B^+: l(\omega_2) \leq \omega_1 \leq L(\omega_2)\}$.

When $m=M$ and the bild is a vertical segment $B = \{m\}\times [y_m,y_M]$ then, clearly, $\cnt(B)=B$ and, in this case, $\cnt^+(B)=B^+=\{m\}\times [0,y_M]$ . If  $m=M$ but the bild is not a vertical line ($\pi_m<\pi_M$) then we claim the center is  $\cnt=\{(m,0)\}$. To see this, first consider that $\vc{z}=(z_1,z_2) \in B$ with $z_1\neq m$. Then ${\vc{z}}^*=(z_1,-z_2) \in B$ and $\tfrac{1}{2}\vc{z}+\tfrac{1}{2}{\vc{z}}^*=(z_1,0) \notin B$. Therefore $\vc{z} \notin \cnt(B)$. It remains to consider the case where $\vc{\omega}=(m,y) \in B$, for some $y\neq 0$. There is $(z_1,z_2), (z_1,-z_2) \in B$ with $z_1 \neq m$ and $z_2 \neq 0$. Assume, without loss of generality that $z_2$ has opposite sign of $y$. Then there is a $\beta \in (0,1)$, such that $ \beta y+(1-\beta) z_2=0$. Clearly, $m \neq \beta m+(1-\beta)z_1 \not \in B \cap \R$, thus
 \[
 \beta(m,y)+(1-\beta)(z_1,z_2)=(\beta m+(1-\beta)z_1 ,0) \neq (m,0)= B\cap\R.
 \]
We concluded that $  \beta(m,y)+(1-\beta)(z_1,z_2) \not\in B$ and therefore that $(m,y) \notin \cnt(B)$, for $y\neq 0$.

In the case where $B = \{m\}\times [y_m,y_M]$, $x_1(y)=m=x_2(y)$ for $y\in [y_m,0]$ and $x_1'(0^-)=x_2'(0^-)=0$, thus $l(y)=L(y)=m$ for any $y \in \R$. Then
\[
\{(x,y) \in B^+: l(y) \leq x \leq L(y)\}=\{(x,y) \in B^+:  x =m\}  =\{m\}\times [0,y_M].
\]
When $m=M$ and $\pi_m<\pi_M$, we know that $x_1(\cdot)\leq x_2(\cdot)$ and $x_1(0)=x_2(0)=m$. Then $a\equiv x_1'(0^-)\geq  x_2'(0^-)\equiv b$. In the case where $a>b$ we have
\[
l(y)=m+ay> m+by=L(y), \text{ for } y > 0.
\]
 Therefore, $\{(x,y) \in B^+: l(y) \leq x\leq L(y)\}=(m,0)$ and this is, in fact, the upper center of $B$.

We now consider $a=b \neq 0$ (the case where $a=b=0$ is the one where $B=\{m\}\times [y_m,y_M]$). Since $x_1(\cdot)$ is convex and $x_2(\cdot)$ is concave we know that, using again \cite[theorem 25.1]{Roc}, $l(y)\leq x_1(y)\leq x_2(y)\leq L(y)$. As a consequence of $a=b$ we have that $l=L$ and thus
 \[
 l(y)= x_1(y)= x_2(y),
 \]
that is, the lower bild is a line, and we can write it as the set
\begin{align*}
B^-&=\{(x,y) \in \R^2: x=l(y), y_m\leq y \leq 0 \}\\
&=\{(x,y) \in \R^2: x=m+ay, y_m\leq y \leq 0 \}.
\end{align*}
Since the upper bild is the conjugate of the lower bild,
\[
B^+=\{(x,y) \in \R^2: x=m-ay, 0\leq y \leq y_M \}.
\]
 Then the intersection of $B^+$ and $l=\{(x,y)\in \R^2:  x=m+ay, y \in \R \}$, when $a\neq 0$ is just $(m,0)$. That is
\[
\{(x,y) \in B^+: x=l(y)\}=(m,0)=\cnt^+(B).
\] \end{proof}

A simple observation on the slope of the lines $l$ and $L$ allows us to give a different proof of the known result of Au-Yeung (see,  \cite[theorem 3]{Ye1}), which establishes an equivalent condition for the convexity of the quaternionic numerical range.

It is well known \cite[theorem 23.1]{Roc} that for any convex function $f$ of real variable and any fixed element $y_1$ in the domain of $f$ the function defined by
\[
y \mapsto \frac{f(y_1)-f(y)}{y_1-y}
\]
is increasing with $y$. Then any line that joins $(f(y),y)$ and $(f(y_1),y_1)$  in the graph of $f$ with $y<y_1$ has slope smaller than $f'(y_1^-)$. Notice now that there is an element $(\pi_m,  y_{\pi_m})$ in the lower bild. Using the previous conclusion when the convex function is $x_1$, the reference point is $y_1=0$ and $x_1(y_{\pi_m})=\pi_m<x_1(0)=m$, we conclude that
\begin{equation}\label{slope_l_1}
 a=x_1'(0^-)\geq \frac{x_1(0)-x_1(y_{\pi_m})}{0-y_{\pi_m}}>0,
\end{equation}
\cite[theorem 23.1]{Roc}, that is, $l$ has positive slope.
For the case when $\pi_m=m$ we have
\begin{eqnarray}
 a=x'_1(0^-) &=& \lim_{\epsilon \rightarrow 0^-} \frac{x_1(0)-x_1(\epsilon)}{0-\epsilon}\nonumber \\
   &=& \lim_{\epsilon \rightarrow 0^-} \frac{m-x_1(\epsilon)}{-\epsilon} \leq 0,\label{slope_l_2}
\end{eqnarray}
since $\epsilon <0$ and $m \leq x_1(\epsilon)$. Thus $l$ has nonpositive slope.

Analogously, it can be shown that when $M<\pi_M$, $L$ has negative slope and when $M=\pi_M$, $L$ has nonnegative slope.

\begin{cor}\label{Cor_AuYeung_result}
The numerical range is convex if and only if $\pi_m=m$ and $\pi_M=M$.
\end{cor}
\begin{proof}
We begin by proving that if $\pi_m\neq m$ or $\pi_M\neq M$, then the numerical range is non-convex. Suppose $\pi_m<m$ (the case $M<\pi_M$ is analogous). Let $x=l(y)=m+ay$ be the left tangent line to $x_1(\cdot)$ at $0$ as in (\ref{tangent lines}).

For $y>0$, we have, by (\ref{slope_l_1}), $x=l(y)=m+ay\geq m$. Notice that $(\pi_m, -y_{\pi_m})\in B^+$. Therefore, $l(-y_{\pi_m})=m+(-y_{\pi_m}) a\geq l(0)=m>\pi_m$. Hence, we have found $(\pi_m, -y_{\pi_m})\in B^+$ such that $l(-y_{\pi_m})>\pi_m$. From theorem \ref{center bild} we have $(\pi_m, -y_{\pi_m})\notin \cnt(B)$ and so $B$ is not convex since $\cnt(B)\neq B$. By theorem \ref{theo convex W and B} we conclude that $W$ is not convex.

Now we prove that if $\pi_m=m$ and $\pi_M=M$ then the numerical range is convex. Recall that $l(y)=ay+m$, with $a\leq 0$, see (\ref{slope_l_2}). For $y>0$, we have $l(y)\leq m$. For every $(x,y)\in B$, we have $x\geq m \geq l(y)$.

Analogously, we can show that $x\leq M\leq L(y)$. From theorem \ref{center bild} we have that $(x,y)\in \cnt(B)$. Since $(x,y)$ is arbitrary, we have that $\cnt(B)=B$ is convex and from theorem \ref{theo convex W and B}, $W$ is convex. \end{proof}

\color{black}

An interesting case, where the center is a kite, is when $\pi_m<m$ and $M<\pi_M$. The next corollary proves this result.

\begin{cor}\label{straightforward characterization}
Let $\pi_m<m\leq M<\pi_M$. Suppose there is $\vc{\tilde{\omega}}\in \C^+$ such that $l\cap L=\{\vc{\tilde{\omega}}\}$. Then,
\[
\cnt(W)\cap \C= \conv \big\{(m,0), (M,0), \vc{\tilde{\omega}}, \vc{\tilde{\omega}}^*\big\} \cap B.
\]
\end{cor}
\begin{proof}
When $\pi_m<m$, as we have noticed in (\ref{slope_l_1}), $l$ has positive slope. Similarly, we can show that $L$ has negative slope. Since $l$ passes through $(m,0)$ and $L$ through $(M,0)$, $l$ and $L$ must cross at a point in $\C^+$. Let this point be $\vc{\tilde{\omega}}$. The result follows from theorem \ref{center bild}.
\end{proof}

The follow example illustrates a case where the center is a kite.

\begin{ex}
Following \cite[page 318]{ST}, let $A=\left[\begin{array}{cc} k_1 i & \alpha \\ -\alpha & 1+k_2 i\end{array}\right],$ with $\alpha, k_1,k_2\in\R^+$ and $\alpha^2>k_1k_2$. In this case, the boundary of the lower bild $B^-$ consists of an ellipse $\mathcal{E}$ and the segment $[m,M]\times\{0\}$, where $(m,0)$ and $(M,0)$ are the points where $\mathcal{E}$ intersects the real axis (the notation in \cite{ST} is $m=T_1$ and $M=T_2$). Our aim is to describe the center of the bild of $A$.

From \cite[lemma 6.4]{ST}, case $5$, the ellipse $\mathcal{E}$ contains the points $(0,-k_1)$, $(1,-k_2)$, $(m,0)$ and $(M,0)$, where
\begin{equation}\label{m and M}
m=\frac{k_1^2}{k_1^2(\alpha+(\alpha^2-k_1k_2)^{\frac{1}{2}})^2}\,\,\,\textrm{ and }\,\,\,M=\frac{k_1^2}{k_1^2(\alpha-(\alpha^2-k_1k_2)^{\frac{1}{2}})^2}.
\end{equation}
Moreover, we know that the vertical lines $x=0$ and $x=1$ are tangent to the ellipse at $(0,-k_1)$ and $(1,-k_2)$, respectively. These data fully characterize the ellipse $\mathcal{E}$. Therefore, if we substitute those points in the general equation
\[
Ax^2+Bxy+Cy^2+Dx+Ey+F=0,
\]
we obtain a homogeneous system of six linear equations with six unknowns. From formulas (\ref{m and M}) one concludes that the linear system's matrix has rank $5$. Solving the linear system leads to the following characterization of $\mathcal{E}$:
\begin{equation}\label{general ellipse eq}
x^2+2(k_2-k_1)\frac{mM}{k_1^2}xy+\frac{mM}{k_1^2}y^2-(M+m)x+\frac{2mM}{k_1}y+mM=0.
\end{equation}
Taking the derivative $\dfrac{d}{dy}$  in (\ref{general ellipse eq}) with $x=x(y)$ (recall that the left derivatives $x_1'(0^-)$ and $x_2'(0^-)$ exist), we get
\begin{equation*}
x_1'(0^-)=\frac{2mM(k_1+(k_2-k_1)m)}{k_1^2(M-m)}\,\,\,\textrm{ and }\,\,\,x_2'(0^-)=-\frac{2mM(k_1+(k_2-k_1)M)}{k_1^2(M-m)}.
\end{equation*}
It is now possible, albeit a tedious computation, to define the lines $l$ and $L$ as in theorem \ref{center bild} and characterize $\cnt(B)$.

Let us consider a more specific example. Take $A=\left[\begin{array}{cc} \frac{1}{8} i & \frac{1}{4} \\ -\frac{1}{4} & 1+\frac{1}{8} i\end{array}\right],$ i.e. $\alpha=\dfrac{1}{4}$ and $k_1=k_2=\dfrac{1}{8}$. Then, the ellipse $\mathcal{E}$ becomes
\[
x^2+4y^2-x+y+\frac{1}{16}=0,
\]
or, in the reduced form,
\[
\frac{(x-\frac{1}{2})^2}{(\frac{1}{2})^2}+\frac{(y+\frac{1}{8})^2}{(\frac{1}{4})^2}=1.
\]
We have:
\[
m=\frac{1}{2}-\frac{\sqrt{3}}{4}\,\,,\,\,M=\frac{1}{2}+\frac{\sqrt{3}}{4}\,\,,\,\,x_1'(0^-)=\frac{2\sqrt{3}}{3}
\,\,,\,\,x_2'(0^-)=-\frac{2\sqrt{3}}{3}.
\]
The lines $l$ and $L$ are given by
\[
l\,:\,\, x=\frac{2\sqrt{3}}{3}y+\frac{1}{2}-\frac{\sqrt{3}}{4}\,\,\,\textrm{ and }\,\,\,L\,:\,\, x=-\frac{2\sqrt{3}}{3}y+\frac{1}{2}+\frac{\sqrt{3}}{4}.
\]
They intersect at $\Big(\dfrac{1}{2}, \dfrac{3}{8}\Big)$, a point on the boundary of the bild of $A$.

We conclude that the center of the bild of $A=\left[\begin{array}{cc} \frac{1}{8} i & \frac{1}{4} \\ -\frac{1}{4} & 1+\frac{1}{8} i\end{array}\right]$ is:
\[
\cnt(B)=\Bigg\{(x,y)\in \R^2: \Big|x-\frac{1}{2}\Big|\leq \frac{\sqrt{3}}{4}-\frac{2\sqrt{3}}{3}|y| \Bigg\}.
\]
\end{ex}


\begin{thebibliography}{99}

\bibitem[AY1]{Ye1} Y. Au-Yeung,\emph{ On the convexity of the numerical range in quaternionic Hilbert space}, Linear and
Multilinear Algebra, \textbf{16} (1984), 93--100.

\bibitem[AY2]{Ye2} Y. Au-Yeung,\emph{ A short proof of a theorem on the numerical range of a normal
quaternionic matrix}, Linear and Multilinear Algebra, \textbf{39:3} (1995), 279--284.

\bibitem[AYS]{AYS} Y. Au-Yeung, L. Siu,\emph{ Quaternionic numerical range and real subspaces}, Linear and Multilinear Algebra, \textbf{45} (1999), 317--327.

\bibitem[CT]{CT}  W. Cheung, N.-K. Tsing, \emph{The $C$-numerical range of matrices is star-shaped}, Linear and Multilinear Algebra, \textbf{41:3} (1996), 245--250.

\bibitem[GR]{GR} K. Gustafson, D. Rao, \emph{Numerical Range}, Springer-Verlag, New York, 1997.

 \bibitem[J]{J} J. Jamison, \emph{Numerical range and numerical radius in quaternionic Hilbert space}, Doctoral Dissertation, Univ. of Missouri, 1972.

\bibitem[Ki]{Ki} R. Kippenhahn, \emph{On the numerical range of a matrix}, Translated from the German by Paul F. Zachlin and Michiel E. Hochstenbach. Linear Multilinear Algebra, \textbf{56:1-2} (2008), 185-225.

\bibitem[K]{K} P. Kumar, \emph{A note on convexity of sections of quaternionic numerical range}, Linear Algebra and its Applications, \textbf{572} (2019), 92-116.

 \bibitem[LLPS18]{LLPS18} P.-S. Lau, C.-K. Li, Y.-T. Poon, N.-S. Sze, \emph{Convexity and Star-shapedness of Matricial Range}, Journal of Functional Analysis, \textbf{275:9} (2018), 2497-2515.

\bibitem[LLPS19]{LLPS19} P.-S. Lau, C.-K. Li, Y.-T. Poon, N.-S. Sze, \emph{The generalized numerical range of a set of matrices}, Linear Algebra and its Applications, \textbf{563} (2019), 24-46.

\bibitem[LNT]{LNT} P.-S. Lau, T.-W. Ng, N.-K. Tsing, \emph{The star-shapedness of a generalized numerical
range}, Linear Algebra and its Applications, \textbf{506} (2016), 308-315.


\bibitem[LP]{LP}  C.-K. Li, Y.-T. Poon, \emph{The joint essential numerical range of operators: Convexity and related results}, Studia Math., \textbf{194} (2009), 91-104.

\bibitem[Roc]{Roc} R. Rockafellar, \emph{Convex Analysis}, Princeton University Press, 1997.

\bibitem[R]{R}   L. Rodman, \emph{Topics in Quaternion Linear Algebra}, Princeton University Press, 2014.

\bibitem[Siu]{Siu} L. Siu, \emph{A study of polynomials, determinants, eigenvalues and numerical ranges over quaternions}, M.Phil. thesis, University of Hong Kong, 1997.

\bibitem[ST]{ST} W. So, R. Thompson, \emph{Convexity of the upper complex plane part of the numerical range of a quaternionic matrix}, Linear and Multilinear Algebra, \textbf{41} (1996), 303--365.

\bibitem[STZ]{STZ} W. So, R. Thompson, F. Zhang, \emph{The numerical range of normal matrices with quaternion entries}, Linear and Multilinear Algebra, \textbf{37} (1994), 175--195.

\bibitem[To]{Thompson} R. Thompson,  \emph{The upper numerical range of a quaternionic matrix is not a complex numerical range}, Linear Algebra and its Applications, \textbf{254:1-3} (1997), 19-28.

\bibitem[Zh]{Zh} F. Zhang, \emph{Quaternions and matrices of quaternions}, Linear Algebra and its Applications, \textbf{251} (1997), 21--57.


\end{thebibliography}
\end{document}